\documentclass{article}

\usepackage{arxiv}
\usepackage{tikz}

\usepackage[utf8]{inputenc} % allow utf-8 input
\usepackage[T1]{fontenc}    % use 8-bit T1 fonts
\usepackage{hyperref}       % hyperlinks
\usepackage{url}            % simple URL typesetting
\usepackage{booktabs}       % professional-quality tables
\usepackage{amsfonts}       % blackboard math symbols
\usepackage{nicefrac}       % compact symbols for 1/2, etc.
\usepackage{microtype}      % microtypography
\usepackage{graphicx}
\usepackage{amsmath}
\usepackage{amsthm}
\newtheorem{thm}{Theorem}[section]
\newtheorem{lem}[thm]{Lemma}
\newtheorem{prop}[thm]{Proposition}
\newtheorem{cor}[thm]{Corollary}
\newtheorem{exmp}{Example}[section]

\title{Fiber-preserving and orientation-reversing involutions of Seifert fibered 3-manifolds}

\author{
  Benjamin Peet\\
  Department of Mathematics\\
  St. Martin's University\\
  Lacey, WA 98503 \\
  \texttt{bpeet@stmartin.edu} \\
}

\begin{document}
\maketitle

\begin{abstract}
We consider fiber-preserving, orientation-reversing involutions on orientable Seifert fibered 3-manifolds and the conditions on a manifold for admissibility of such involutions. We construct a class $\Psi$ of fiber-preserving, orientation-reversing involutions that act trivially on the base. Each element of $\Psi$ is obtained by extending a product involution across Seifert pieces of type $V(2,2;-1)$ - a solid torus with three fibers filled according to seifert invariants $(2,1)$, $(2,1)$, and $(1,-1)$. We show that $\Psi$ forms a single conjugacy class under fiber-preserving diffeomorphisms. Our main result establishes that any fiber-preserving, orientation-reversing involution factors as $\psi\circ g$, where $g$ is fiber-preserving and orientation-preserving and $\psi\in\Psi$, thus reducing the problem to the previously known orientation-preserving case. Through the orientable base-space double covering, we further extend the classification to manifolds with non-orientable base orbifold.
\end{abstract}

% keywords can be removed

\section{Introduction}

In three previous papers \cite{peet2019finite}, \cite{peet2022finite}, and \cite{peet2019finitenon}, we examined the actions of finite, fiber-preserving, and orientation-preserving group actions on orientable Seifert manifolds. In particular, the following result was established:

\begin{thm}
Let $M$ be a closed, compact, and orientable Seifert $3$-manifold that fibers over an orientable base space. Let $\varphi:G\rightarrow Diff_{+}^{fp}(M)$ be a finite group action on $M$ such that the obstruction class can be expressed as $$b=\sum_{i=1}^{m}(b_{i}\cdot\#Orb_{\varphi}(\alpha_{i}))$$ for a collection of fibers $\{\alpha_{1},\ldots,\alpha_{m}\}$ and integers $\{b_{1},\ldots,b_{m}\}$. Then $\varphi$ is an extended product action.
\end{thm}

In essence, this result established that given the condition on the obstruction term, finite actions could be expressed as the extension of a product action across Dehn fillings for some product structure $S^1 \times F$. 

This paper considers orientation-reversing involutions of orientable Seifert manifolds. 

The first part of the paper looks at the admissible Seifert manifolds, summarizing some of the previous work in \cite{peet2019finite} and \cite{peet2019finitenon}. In particular, establishing that the conditions on such a manifold to permit such an involution are that the invariants must be of the form $M=(g,o_1|(2,1),\ldots,(2,1),(1,-\frac{n}{2}))$ where there are $n$ critical fibers of order 2 with $n$ even.

We then look at the work of \cite{dugger2019involutions} to consider the possible involutions of surfaces of genus $g$. In particular, we establish conjugacy classes of such involutions.

We then analyze the form of an orientation-reversing involution, offering a method to construct a class of involutions (given an admissible manifold) that are the identity on the base space by extending a product involution across multiple copies of $V(2,2;-1)$. Here $V(2,2;-1)$ is a  trivially fibered tori with fixed product structure on the boundary and then three regular fibers drilled out and filled according to the Seifert invariants $(2,1)$, $(2,1)$, and $(1,-1)$. We denote $\Psi$ for elements of the class and show that it is a conjugacy class.

The main result reduces the problem to Theorem 1.1 (in the case of involutions) and states that any fiber-preserving and orientation-reversing involution is a fiber-preserving and orientation-preserving involution composed with an element of $\Psi$:

\newtheorem*{thm:mainresult}{Theorem \ref{thm:mainresult}}
\begin{thm:mainresult}
    Let $M$ be a closed, compact, and orientable Seifert $3$-manifold that fibers over an orientable base space and that is not $S^1 \times S$ for some surface $S$. Then, all fiber-preserving and orientation-reversing involutions can be expressed as $\psi\circ g$ for $g$ a fiber-preserving and orientation-preserving involution and $\psi\in\Psi$.

\end{thm:mainresult}

We then use the result of \cite{peet2019finitenon} that:

\newtheorem*{thm:nonorientable}{Theorem \ref{thm:nonorientable}}
\begin{thm:nonorientable}
 Let $M$ be an orientable Seifert fibered manifold that fibers over an orbifold $B$ that has non-orientable underlying space and $p:\tilde{M}\rightarrow M$ be the orientable base space double cover. Let the covering translation be $\tau:\tilde{M}\rightarrow\tilde{M}$. Then there exists an isomorphism between $Diff_{+}^{fp}(M)$ and $Cent_{+}^{fop}(\tau)$.
\end{thm:nonorientable}

In particular, this says that given an involution of $f$ on $M$ we can lift to an involution $\tilde{f}$ on $\tilde{M}$ so that $p\circ\tilde{f}=f\circ p$.

We extend this to see an isomorphism between $Diff_{-}^{fp}(M)$ and $Cent_{-}^{for}(\tau)$. This means that given an orientation-reversing involution on $M$ we can lift to an orientation-reversing involution on $\tilde{M}$, the form of which is now given.

We finish with examples of the involutions on a given admissible manifold.

\section{Preliminary definitions}

We first give some preliminary definitions.
Let $M$ be an oriented smooth manifold of dimension $3$ without boundary. We let $Diff(M)$ be the group of self-diffeomorphisms of $M$ and use the notation $Diff_{-}(M)$ for the group of orientation-reversing self-diffeomorphisms of $M$. An involution is an element of order 2 of $Diff(M)$.

$M$ will be assumed to be an orientable Seifert-fibered manifold. We use the original Seifert definition. That is, a Seifert manifold is a 3-manifold such that $M$ can be decomposed into disjoint fibers where each fiber is a simple closed curve. Then for each fiber $\gamma$, there exists a fibered neighborhood (that is, a subset consisting of fibers and containing $\gamma$) which can be mapped under a fiber-preserving map onto a solid fibered torus. \cite{Seifert1933}

We call a Seifert bundle a Seifert manifold $M$ along with a continuous map $\pi:M\rightarrow B$ where $\pi$ identifies each fiber to a point. 

 We use the notation $Diff^{fp}(M)$ for the group of fiber-preserving self-diffeomorphisms of $M$ (given some Seifert fibration). Given a fiber-preserving diffeomorphism, there is an induced diffeomorphism on the underlying space $B_{U}$ of the base space $B$. Given  $f\in Diff^{fp}_{-}(M)$, it follows from the Seifert definition that either the orientation of the fibers is reversed or the induced map on the underlying space is orientation-reversing.

If we have a manifold $M$, then a product structure on $M$ is a diffeomorphism $k:A\times B\rightarrow M$ for some manifolds $A$ and $B$. \cite{lee2003smooth} If a Seifert-fibered manifold $M$ has a product structure $k:S^{1}\times F\rightarrow M$ for some surface $F$ and $k(S^{1}\times\{x\})$ are the fibers of $M$ for each $x\in F$, then we say that $k:S^{1}\times F\rightarrow M$ is a fibering product structure of $M$. 

We note here that a fibering product structure on $M$ is equivalent to the existence of a foliation of $M$ by both circles and by surfaces diffeomorphic to $F$ so that any circle intersects each foliated surface exactly once.

Given that the first homology group (equivalently the first fundamental group) of a torus is $\mathbb{Z}\times\mathbb{Z}$ generated by two elements represented by any two nontrivial loops that cross at a single point, we can use the meridian-longitude framing from a product structure as representatives of two generators. If we have a diffeomorphism $f:T_{1}\rightarrow T_{2}$ and product structures $k_{i}:S^{1}\times S^{1}\rightarrow T_{i}$, then we can express the induced map on the first homology groups by a matrix that uses bases for $H_{1}(T_{i})$ derived from the meridian-longitude framings that arise from $k_{i}:S^{1}\times S^{1}\rightarrow T_{i}$.

We say that a $f \in Diff(A\times B)$ is a product diffeomorphism if $f:A\times B\rightarrow A\times B$ can be expressed as $(f_1,f_2$ where $f_1:A\rightarrow A$ and $f_2:B\rightarrow B$. 

Then, given $f\in Diff(M)$ and a product structure $k:A\times B\rightarrow M$, we say that $f$ leaves the product structure  invariant if $k^{-1}\circ f \circ k$ defines a product diffeomorphism on  $A\times B$.

Suppose that now we have a fibering product structure $k:S^{1}\times F\rightarrow M$. We then say that each boundary torus is positively oriented if the fibers are given an arbitrary orientation, and then each boundary component of $k(\{u\}\times F)$ is oriented by taking the normal vector to the surface according to the orientation of the fibers.

We now consider under what circumstances a fiber-preserving, but orientation-reversing involution is admitted.

\section{Admissible Seifert manifolds}

\subsection{Previous results}
We begin with some preliminary results. All of these were published in \cite{peet2022finite}

\begin{lem}Let $F$ be an orientable surface with boundary. Let the boundary be positively oriented according to some orientation of $F$ and $f:F\rightarrow F$ be a diffeomorphism. Then $f$ is orientation-preserving on $F$ if and only if $f$ is orientation-preserving between some pair of boundary components.
\end{lem}

\begin{cor} Let $\hat{M}$ be an oriented trivially Seifert fibered 3-manifold with positively oriented boundary $\partial\hat{M}=T_{1}\cup\ldots\cup T_{n}$. Then a fiber-preserving diffeomorphism $f:\hat{M}\rightarrow\hat{M}$ is orientation-preserving if and only if $f$ is orientation-preserving between some pair of boundary tori.
\end{cor}

\begin{prop} All finite, fiber-preserving actions on an orientable Seifert 3-manifold fibering over an orientable base space with at least one critical fiber of order greater than two are orientation-preserving.
\end{prop}

\begin{prop} All finite, fiber-preserving actions on an orientable Seifert 3-manifold fibering over an orientable base space with nonzero Euler class are orientation-preserving.
\end{prop}

\subsection{Fixed point free involutions}

We now now consider the admissibility of such involutions that are fixed point free.

If such an involution $f$ existed for a Seifert manifold $\tilde{M}$, then the quotient space $M=\tilde{M}/f$ would be a nonorientable Seifert manifold. Then necessarily the obstruction term of $\tilde{M}$ must be zero, which is only the case when $\tilde{M} \cong S^1 \times S$ for a closed surface $S$. See \cite{scott1983geometries} for more details.

It then follows that there are no fixed point free involutions for Seifert manifolds not homeomorphic to $S^1 \times S$.

\subsection{Summary of admissible Seifert manifolds}
To summarize, these results establish that if there is an orientation-reversing, fiber-preserving involution on a Seifert manifold $M=(g,o_{1}|(q_{1},p_{1}),\ldots(q_{n},p_{n}),(1,b))$ then necessarily:

\begin{enumerate}
\item $e=-(b+\sum\limits_{i=1}^{n}\frac{p_{i}}{q_{i}})=0$

\item $q_{i}=2$ for all $i=1,\ldots,n$

\item $n$ is even

\item $b=-\frac{n}{2}$ In particular $b=0$ if there are no critical fibers and so $M=S^{1}\times S$ for some surface.

\end{enumerate}

Given \cite{scott1983geometries} we note that the only possible geometries are those of $S^{2}\times \mathbb{R}$, $E^{3}$, and $H^{2}\times \mathbb{R}$ depending upon the Euler characteristic of the base space orbifold $B$ which we begin by considering as orientable only.

Now using the Riemann-Hurwitz formula, we note that given each $q_{i}=2$:

$$\chi_{orb}(B)=\chi(B_{U})-\sum_{i=1}^{n}(1-\frac{1}{2})$$

This simplifies to

$$\chi_{orb}(B)=\chi(B_{U})-\frac{n}{2}$$

Hence, we consider the following three cases:

\begin{enumerate}
    \item $\chi(B_{U})>\frac{n}{2}$ (Geometry of $S^{2}\times \mathbb{R}$)
    \item $\chi(B_{U})=\frac{n}{2}$ (Geometry of $E^{3}$)
    \item $\chi(B_{U})<\frac{n}{2}$ (Geometry of $H^{2}\times \mathbb{R}$)
\end{enumerate}

These cases reduce to the following:

\begin{enumerate}
    \item $B_{U}=S^{2}$ and $n=0$ or $n=2$ ($S^{2}\times \mathbb{R}$)
    \item $B_{U}=S^{2}$ and $n=4$ or $B_{U}=T^{2}$ and $n=0$ ($E^{3}$)
    \item $B_{U}$ is hyperbolic and $n$ is any even integer or $B_{U}=T^{2}$ and $n$ is any even integer greater than 0 or $B_{U}=S^{2}$ and $n$ is any even integer greater than 2 ($H^{2}\times \mathbb{R}$)
\end{enumerate}

Given that $b=-\frac{n}{2}$, we can use the Seifert notation to list these cases as:

\begin{enumerate}
    \item 
    \begin{enumerate}
        \item $M=(0,o_{1}|)$
        \item $M=(1,o_{1}|(2,1),(2,1),(1,-1))$
    \end{enumerate}
    
    Both are $S^{2}\times S^{1}$ but are either fibered trivially or fibered with two fibers of order 2.
    \item 
    \begin{enumerate}
        \item $M=(0,o_{1}|(2,1),(2,1),(2,1),(2,1),(1,-2))$
        \item $M=(1,o_{1}|)$
    \end{enumerate}
    The second is simply $T^{3}$ fibered trivially

    \item 
    \begin{enumerate}
        \item $M=(g,o_{1}|(2,1),\dots,(2,1),(1,-\frac{n}{2})$
and $n$ is any even non-negative integer, and there are $n$ critical fibers. (Note that this includes $(g,o_{1}|)$)
        \item $M=(1,o_{1}|(2,1),\dots,(2,1),(1,-\frac{n}{2})$ and $n$ is any even integer greater than 0 and there are $n$ critical fibers.
        \item $M=(0,o_{1}|(2,1),\dots,(2,1),(1,-\frac{n}{2})$ and $n$ is any even integer greater than 2 and there are $n$ critical fibers.
    \end{enumerate}
\end{enumerate}

\section{Involutions on genus $g$ surfaces}

We use this section to summarize the work of \cite{dugger2019involutions}. The primary result we utilize is Theorem 5.7.

To do so, we define the three classes of orientation-preserving involutions and the two classes of orientation-reversing involutions. All are distinct up to conjugacy by a diffeomorphism of the surface. We describe the involutions, but for a more thorough treatment with images, see \cite{dugger2019involutions}.

\subsection{Orientation-preserving}

The first involution is the trivial involution. The second type is:

$$spit_{g,r}:S_g \rightarrow S_g$$

Which a rotation about an axis that fixes $g-2r$ holes and exchanges $r$ holes.

This involution has $4(g-2r)$ fixed points.

The third type is defined only when $g$ is odd:

$$rot:S_g\rightarrow S_g$$

This is a rotation about an axis that passes through the "center" hole.

This involution has no fixed points.

\subsection{Orientation-reversing}

The first class is the reflections that fix $r$ holes of the genus $g$ surface and exchange the others. We denote these as:

$$refl_{g,r}:S_g\rightarrow S_g$$

This involution has $g-2r+1$ disjoint circles of fixed points.

The second class is constructed by taking an antipodal map on a genus $g-r\leq g$ surface with $r\geq0$ handles. We denote these as:

$$anti_{g,r}:S_g \rightarrow S_g$$

This involution has $r$ disjoint circles of fixed points. If $r=0$ there are no fixed points.

\subsection{Theorem}

We state here the full theorem for completeness:

\begin{thm}
    For $g\geq0$ there are $4+2g$ conjugacy classes of involutions on the genus $g$ surface. 
    
    These are the $2+\lceil \frac{g}{2} \rceil$ orientation-preserving involutions:

    \begin{enumerate}
        \item $[id]$
        \item $[spit_{g,r}]$ for $0\leq r\leq \frac{g}{2}$
        \item $[rot]$ when $g$ is odd
    \end{enumerate}

and the $2+g+\lfloor \frac{g}{2} \rfloor$ orientation-preserving involutions:

    \begin{enumerate}
        \item $[refl_{g,r}]$ for $0\leq r\leq \frac{g}{2}$
        \item $[anti_{g,r}]$ for $0\leq r \leq g$
    \end{enumerate}
\end{thm}

As an example, here we take the torus:

\begin{exmp}
    As $g=1$ there are 3 orientation-preserving and 3 orientation-preserving involutions up to conjugation. As to the orientation-reversing involutions. Within Example 10.2 of \cite{dugger2019involutions}, we have that for the orientation-reversing involutions, up to conjugation, the induced involutions on $H_1(T)$ given by:
    $$\left[\begin{array}{cc}
1 & 0\\
0 & -1
\end{array}\right]$$ for $[anti_{1,0}]$ and $[refl_{1,0}]$
and $$\left[\begin{array}{cc}
0 & 1\\
1 & 0
\end{array}\right]$$ for $[anti_{1,1}]$.
    
\end{exmp}

Now, by Section 3.2, we can disregard the fixed-point free involutions $rot$ for $g$ odd and $anti_{g,0}$ from consideration in this paper.

\section{Preliminary results}

We now prove the following propositions:

\begin{prop}
Let $F$ be a surface with a single boundary component and let $\hat{M}\cong S^1\times F$ be fibered according to a fibering product structure $k_{\hat{M}}:S^1\times F\rightarrow \hat{M}$. Suppose $M$ is obtained by a $(1,2)$ Dehn filling of the boundary component. If $f_{\hat{M}}:\hat{M}\rightarrow \hat{M}$ is a fiber-preserving and orientation-reversing involution then it extends over the fillings to a fiber-preserving, orientation-reversing involution on $M$ if and only if:

$$((k_{\hat{M}}^{-1} \circ f_{\hat{M}} \circ k_{\hat{M}})|_{S^1 \times \partial F})_*=\pm \left[\begin{array}{cc}
1 & -1\\
0 & -1
\end{array}\right]$$
\end{prop}

\begin{proof}

    We first note that the filling torus $V$ will have an induced $(-2,1)$ fibration according to some product structure $k_V:S^1\times D \rightarrow V$. See \cite{peet2019finite} for more details. Hence to extend within this solid torus, the map will need to preserve the meridian - this follows from the half-lives half-dies theorem, see Lemma 3.5 of \cite{hatcher2000notes} - but also the fibration. The map must either reverse the orientation of the fibers or the orientation of a meridian (but not both).
    
    Hence if we represent the map on $\partial V$ homologically (according to the induced product structure from $k_V$) as $\left[\begin{array}{cc}
a & b\\
c & d
\end{array}\right]$ then:
$$\left[\begin{array}{cc}
a & b\\
c & d
\end{array}\right]\left[\begin{array}{cc}
0 \\
1 
\end{array}\right]=\varepsilon\left[\begin{array}{cc}
0 \\
-1 
\end{array}\right]
$$

and 

$$\left[\begin{array}{cc}
a & b\\
c & d
\end{array}\right]\left[\begin{array}{cc}
-2 \\
1 
\end{array}\right]=\varepsilon\left[\begin{array}{cc}
-2 \\
1 
\end{array}\right]$$

Where $\varepsilon=\pm 1$.

The only solutions to this are $$\pm \left[\begin{array}{cc}
1 & 0\\
-1 & -1
\end{array}\right]$$

We then see that the following homological diagram (according to the induced product structures on the boundaries from $k_{\hat{M}}$ and $k_V$) is yielded:

$$\begin{array}{ccccc}
 &  & \left[\begin{array}{cc}
0 & 1\\
1 & 2
\end{array}\right]\\
 & H_{1}(\partial \hat{M}) & \leftarrow & H_{1}(\partial V)\\
 & \downarrow &  & \downarrow & \pm \left[\begin{array}{cc}
1 & 0\\
-1 & -1
\end{array}\right]\\
 & H_{1}(\partial\hat{M}) & \leftarrow & H_{1}(\partial V)\\
 &  & \left[\begin{array}{cc}
0 & 1\\
1 & 2
\end{array}\right]
\end{array}$$

Hence, $$((k_{\hat{M}}^{-1} \circ f_{\hat{M}} \circ k_{\hat{M}})|_{S^1 \times \partial F})_*=\pm \left[\begin{array}{cc}
1 & -1\\
0 & -1
\end{array}\right]$$

For the converse, note that if $\partial\hat{M}$ is filled with filling map $d:\partial V \rightarrow \partial \hat{M}$, then the map:

$$d^{-1} \circ f_{\hat{M}}|_{\partial \hat{M}} \circ d$$

Sends the boundary of $V$ to the boundary of $V$. It remains to show that this map can be extended within the solid torus to a map that is both fiber-preserving and an involution.

To do so, we note that:

$$((k_V|_{S^1 \times \partial D})^{-1} \circ(d^{-1} \circ f_{\hat{M}}|_{\partial \hat{M}} \circ d)\circ (k_V|_{S^1 \times \partial D}))_*=\pm \left[\begin{array}{cc}
1 & 0\\
-1 & -1
\end{array}\right]$$

So then by Propositions 10.5 and 10.6 of \cite{dugger2019involutions}, this is conjugate in $GL_2(\mathbb{Z})$ to $\left[\begin{array}{cc}
0 & 1\\
1 & 0
\end{array}\right]$.

Hence, we must have that $(k_V|_{S^1 \times \partial D})^{-1} \circ(d^{-1} \circ f_{\hat{M}}|_{\partial \hat{M}} \circ d)\circ (k_V|_{S^1 \times \partial D})$ is conjugate to $anti_{1,1}$ via some $h_1\in Diff^{+}(S^1\times \partial D)$ and in particular we can take it to be without loss of generality by an orientation-preserving diffeomorphism (if not take $h_1\circ anti_{1,1}$) with homological representative either $\pm\left[\begin{array}{cc}
-1 & -1\\
1 & 0
\end{array}\right]$. See \cite{dugger2019involutions} for the algebraic calculations. So then:

$$anti_{1,1}=h_1^{-1} \circ ((k_V|_{S^1 \times \partial D})^{-1} \circ(d^{-1} \circ f_{\hat{M}}|_{\partial \hat{M}} \circ d)\circ (k_V|_{S^1 \times \partial D})) \circ h_1$$

We use here the fact that the mapping class group of the torus is in one-to-one correspondence with $GL_2(\mathbb{Z})$. See \cite{stillwell2012classical} for a thorough proof.

However, considering the maps $g\in Diff(S^1\times \partial D)$ with $g_{\pm}(u,v)=(u^{\pm 1},u^{\mp 1}v^{\mp 1})$ we note that $g_{\pm}*=\pm\left[\begin{array}{cc}
1 & 0\\
-1 & -1
\end{array}\right]$ and by a similar argument to above, there is a $h_2 \in Diff^{+}(S^1 \times \partial D)$ such that:

$$anti_{1,1}=h_2^{-1} \circ g \circ h_2$$

So in particular:

$$g=h_2\circ h_1^{-1} \circ [(k_V|_{S^1 \times \partial D})^{-1} \circ(d^{-1} \circ f_{\hat{M}}|_{\partial \hat{M}} \circ d)\circ (k_V|_{S^1 \times \partial D})] \circ h_1 \circ h_2^{-1}$$

So now $(k_V|_{S^1 \times \partial D})^{-1} \circ(d^{-1} \circ f_{\hat{M}}|_{\partial \hat{M}} \circ d)\circ (k_V|_{S^1 \times \partial D)}$ is conjugate to $g$ via an orientation-preserving map $h=h_1 \circ h_2^{-1}$, and in particular:

$$h_*=\left[\begin{array}{cc}
1 & 1\\
-1 & 0
\end{array}\right](\pm\left[\begin{array}{cc}
1 & 1\\
-1 & 0
\end{array}\right])^{-1}=\pm I$$

Then $(k_V|_{S^1 \times \partial D})^{-1} \circ(d^{-1} \circ f_{\hat{M}}|_{\partial \hat{M}} \circ d)\circ (k_V|_{S^1 \times \partial D})$ is conjugate to $g$ via a conjugating map that is isotopic to plus or minus the identity, and importantly $k_V \circ g\circ k_V^{-1}$ is fiber-preserving (fibers are given by $k_V(S^{1}\times {u})$) and can be extended within the solid torus by $\overline{g}(u,\rho v)=(u,\rho u^{-1}v^{-1})$. This is still fiber-preserving and an involution.

Let the Seifert manifold $M'$ be obtained by filling according to $d \circ h$. This is isomorphic to $M$ by some fiber-preserving diffeomorphism $\varphi:M\rightarrow M'$ which is the identity when restricted to $\hat{M}$.

$f$ can be extended to a fiber-preserving involution on $M'$ as $\overline{f'}$ and then $\overline{f}=\varphi \circ \overline{f'} \circ \varphi^{-1}$ is $f$ when restricted to $\hat{M}$ and still a fiber-preserving involution.

\end{proof}

\begin{prop}
Let $F$ be a surface with a single boundary component and let $\hat{M}\cong S^1\times F$ be fibered according to a fibering product structure $k_{\hat{M}}:S^1\times F\rightarrow \hat{M}$. Suppose $M$ is obtained by a $(x,1)$ Dehn filling of the boundary component for some $x\in \mathbb{Z}$. If $f_{\hat{M}}:\hat{M}\rightarrow \hat{M}$ is a fiber-preserving and orientation-reversing involution then it extends over the fillings to a fiber-preserving, orientation-reversing involution on $M$ if and only if:

$$((k_{\hat{M}}^{-1} \circ f_{\hat{M}} \circ k_{\hat{M}})|_{S^1 \times \partial F})_*=  \pm \left[\begin{array}{cc}
1 & -2x\\
0 & -1
\end{array}\right]$$
\end{prop}

\begin{proof}
The proof mimics the previous proposition. 

This time the filling tori will have induced $(1,0)$ fibrations. Again, the map will need to preserve the meridian and the fibration reversing orientation of one but not both. Hence:
$$\left[\begin{array}{cc}
a & b\\
c & d
\end{array}\right]\left[\begin{array}{cc}
0 \\
1 
\end{array}\right]=\varepsilon\left[\begin{array}{cc}
0 \\
-1 
\end{array}\right]
$$

and 

$$\left[\begin{array}{cc}
a & b\\
c & d
\end{array}\right]\left[\begin{array}{cc}
1 \\
0 
\end{array}\right]=\varepsilon\left[\begin{array}{cc}
1 \\
0 
\end{array}\right]$$

Here $\varepsilon=\pm 1$.

The only solution to this is:
$$\pm \left[\begin{array}{cc}
1 & 0\\
0 & -1
\end{array}\right]$$

We then see that the following homological diagram is yielded:

$$\begin{array}{ccccc}
 &  & \left[\begin{array}{cc}
-1 & x\\
0 & 1
\end{array}\right]\\
 & H_{1}(\partial \hat{M}) & \leftarrow & H_{1}(\partial V)\\
 & \downarrow &  & \downarrow & \pm \left[\begin{array}{cc}
1 & 0\\
0 & -1
\end{array}\right]\\
 & H_{1}(\partial\hat{M}) & \leftarrow & H_{1}(\partial V)\\
 &  & \left[\begin{array}{cc}
-1 & x\\
0 & 1
\end{array}\right]
\end{array}$$

Hence: $$((k_{\hat{M}}^{-1} \circ f_{\hat{M}} \circ k_{\hat{M}})|_{S^1 \times \partial F})_*=\pm\left[\begin{array}{cc}
1 & -2x\\
0 & -1
\end{array}\right] $$

We employ Proposition 10.6 of \cite{dugger2019involutions} once again to yield some $h\in Diff(S^1 \times \partial D)$ such that:

$$g=h^{-1} \circ ((k_V|_{S^1 \times \partial D})^{-1} \circ(d^{-1} \circ f_{\hat{M}}|_{\partial \hat{M}} \circ d)\circ (k_V|_{S^1 \times \partial D} ))\circ h$$

Where $g(u,v)=(u^{\pm 1}, v^{\mp 1})$.

Once again, we can assume $h\in Diff^+(S^1\times \partial D)$ and that it is isotopic to plus or minus the identity by a simple calculation and composing with a $refl_{1,0}$ if orientation-reversing.

So now $(k_V|_{S^1 \times \partial D})^{-1} \circ(d^{-1} \circ f_{\hat{M}}|_{\partial \hat{M}} \circ d)\circ (k_V|_{S^1 \times \partial D}) $ is conjugate to $refl_{1,0}$ via a conjugating map that is isotopic to plus or minus the identity. $(k_V|_{S^1\times \partial D})^{-1} \circ refl_{1,0} \circ (k_V|_{S^1\times \partial D})$ is certainly fiber-preserving (fibers are given by $k_V(S^{1}\times {u})$) can be extended within the solid torus by $\overline{g}(u,\rho v)=(u^{\pm 1},\rho v^{\mp 1})$. This is still fiber-preserving and an involution.

Once more, the Seifert manifold $M'$ obtained by filling according to $d \circ h$ is isomorphic to $M$ by a fiber-preserving diffeomorphism $\varphi:M\rightarrow M'$ which is the identity when restricted to $\hat{M}$.

$f$ can be extended to a fiber-preserving involution on $M'$ as $\overline{f'}$ and then $\overline{f}=\varphi \circ \overline{f'} \circ \varphi^{-1}$ is $f$ when restricted to $\hat{M}$ and still a fiber-preserving involution.

\end{proof}

\section{Construction of the fiber- and orientation-reversing involutions $\Psi$}

This section lays out the construction of a class of fiber- and orientation-reversing involutions that we denote $\Psi$. These involutions are all the identity on the base space.

\subsection{Construction of $V(2,2;-1)$}

We first consider the construction of $V(2,2;-1)$. This is constructed by taking a solid torus $V$ with product structure $k_V:S^1\times D \rightarrow V$ and trivial fibering given by fibers $k_V(S^1\times\{x\})$ and then drilling out fibered torus neighborhoods of three interior fibers and gluing in three solid tori according to the Seifert invariants $(2,1)$,$(2,1)$, and $(1,-1)$.

More precisely, let $\hat{V}$ be $V$ with three fibered torus neighborhoods of interior fibers drilled out have fibering product structure $k_{\hat{V}}:S^1\times (D\setminus(D_1 \cup D_2 \cup D_3))\rightarrow \hat{V}$ and let the inner torus boundaries be denoted $T_1$,$T_2$, and $T_3$ with the outer denoted $T$. Then if $X=V_1\cup V_2 \cup V_3$  is a disjoint union of three solid tori, then we have a gluing map $d:\partial X \rightarrow \partial \hat{V}$ and some product structure $k_X:S^1 \times (D_1\cup D_2 \cup D_3)\rightarrow X$ on $X$ so that the restricted positively oriented product structures $k_{\partial \hat{V}_i}: S^1\times S^1\rightarrow \partial \hat{V}_i$ and $k_{T_i}:S^1 \times S^1 \rightarrow T_i$ obey the following:

$$(k_{T_i}^{-1} \circ d|_{\partial V_i} \circ k_{\partial V_i})(u,v)=(v,uv^2)$$ for $i=1,2$

$$(k_{T_3}^{-1} \circ d|_{\partial V_3} \circ k_{\partial V_3})(u,v)=(u^{-1}v^{-1},v)$$ for $i=3$

Note that the boundary torus of this $V(2,2;-1)$ still has the restricted positively oriented fibering product structure $k_T:S^1 \times S^1 \rightarrow T$.

\subsection{Construction of $M=(g,o_1|(2,1),\ldots,(2,1),(1,-\frac{n}{2}))$}

It now follows that $M=(g,o_1|(2,1),\ldots,(2,1),(1,-\frac{n}{2}))$ can be yielded by viewing it as $M=(g,o_1|(2,1),\ldots,(2,1),(1,-1),\ldots,(1,-1))$, that is, $n$ $(2,1)$ fillings and $\frac{n}{2}$ $(1,-1)$ fillings. Then taking $\hat{M}$ with a fibering product structure $k_{\hat{M}}:S^1 \times F \rightarrow \hat{M}$ and $F$ a genus $g$ surface with filling each boundary torus with a $V(2,2;-1)$ according to a trivial filling.

That is, for a boundary component $(\partial \hat{M})_i$ with restricted, positively oriented fibering product structure $k_{(\partial \hat{M})_i}:S^1 \times S^1 \rightarrow (\partial \hat{M})_i$ we can fill using some  $d:T\rightarrow (\partial \hat{M})_i$ so that:

$$(k_{(\partial \hat{M})_i}^{-1} \circ d \circ k_T)(u,v)=(u,v)$$

\subsection{Construction of fiber- and orientation-reversing involution}

We pick a fiber-preserving and orientation-reversing involution $\psi_{\hat{M}}\in  Diff(\hat{M})$ by: 
$$(k_{\hat{M}}^{-1}\circ \psi_{\hat{M}} \circ k_{\hat{M}})(u,x)=(u^{-1},x)$$

\subsection{Fiber-preserving and orientation reversing involution on $V(2,2;-1)$}

For convenience we use $S^1\times((3I\times3I)\setminus(D_1\cup D_2\cup D_3))$ in place of $S^1\times (D\setminus(D_1 \cup D_2 \cup D_3))$, that is, the product structure on $\hat{V}$ is now $k_{\hat{V}}:S^1\times ((3I\times3I))\setminus(D_1 \cup D_2 \cup D_3)) \rightarrow \hat{V}$

\tikzset{every picture/.style={line width=0.75pt}}
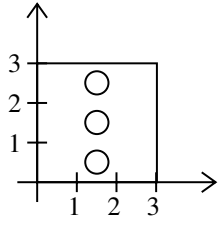
\begin{figure}
 %set default line width to 0.75pt        
\centering
\begin{tikzpicture}[x=0.75pt,y=0.75pt,yscale=-1,xscale=1]

%Shape: Rectangle [id:dp14086199351207807] 
\draw   (139.75,90) -- (200.14,90) -- (200.14,150) -- (139.75,150) -- cycle ;
%Shape: Axis 2D [id:dp7123114601867097] 
\draw  (129.75,150) -- (229.75,150)(139.75,60) -- (139.75,160) (222.75,145) -- (229.75,150) -- (222.75,155) (134.75,67) -- (139.75,60) -- (144.75,67) (159.75,145) -- (159.75,155)(179.75,145) -- (179.75,155)(199.75,145) -- (199.75,155)(134.75,130) -- (144.75,130)(134.75,110) -- (144.75,110)(134.75,90) -- (144.75,90) ;
\draw   (166.75,162) node[anchor=east, scale=1]{1} (186.75,162) node[anchor=east, scale=1]{2} (206.75,162) node[anchor=east, scale=1]{3} (136.75,130) node[anchor=east, scale=1]{1} (136.75,110) node[anchor=east, scale=1]{2} (136.75,90) node[anchor=east, scale=1]{3} ;
%Shape: Circle [id:dp6731096328107613] 
\draw   (164,139.86) .. controls (164,136.62) and (166.62,134) .. (169.86,134) .. controls (173.09,134) and (175.71,136.62) .. (175.71,139.86) .. controls (175.71,143.09) and (173.09,145.71) .. (169.86,145.71) .. controls (166.62,145.71) and (164,143.09) .. (164,139.86) -- cycle ;
%Shape: Circle [id:dp2120951009217109] 
\draw   (164,99.86) .. controls (164,96.62) and (166.62,94) .. (169.86,94) .. controls (173.09,94) and (175.71,96.62) .. (175.71,99.86) .. controls (175.71,103.09) and (173.09,105.71) .. (169.86,105.71) .. controls (166.62,105.71) and (164,103.09) .. (164,99.86) -- cycle ;
%Shape: Circle [id:dp46496429862366195] 
\draw   (164,119.86) .. controls (164,116.62) and (166.62,114) .. (169.86,114) .. controls (173.09,114) and (175.71,116.62) .. (175.71,119.86) .. controls (175.71,123.09) and (173.09,125.71) .. (169.86,125.71) .. controls (166.62,125.71) and (164,123.09) .. (164,119.86) -- cycle ;

\end{tikzpicture}
\caption{$3I\times3I$ less three discs}
\end{figure}

We define $\psi_{\hat{V}}$ by:

$$k^{-1}_{\hat{V}}\circ \psi_{\hat{V}}\circ k_{\hat{V}}(u,x,1)=(u^{-1}e^{-2\pi ix},x,1)$$
$$k^{-1}_{\hat{V}}\circ \psi_{\hat{V}}\circ k_{\hat{V}}(u,x,2)=(u^{-1}e^{2\pi ix},x,2)$$

and then:

$$k^{-1}_{\hat{V}}\circ \psi_{\hat{V}}\circ k_{\hat{V}}(u,x,y)=(u^{-1},x,y)$$ 

if either $x=0,1$ or $y=0,1$. That is, on the boundary of $\hat{V}$.

Then by coning inwards towards the centers of each of the removed discs we see that homologically according to the framing given by the product structure we have:
$$((k^{-1}_{\hat{V}}\circ \psi_{\hat{V}}\circ k_{\hat{V}})|_{\partial D_1})*=\left[\begin{array}{cc}
-1 & 1\\
0 & 1
\end{array}\right]$$ $$((k^{-1}_{\hat{V}}\circ \psi_{\hat{V}}\circ k_{\hat{V}})|_{\partial D_2})*=\left[\begin{array}{cc}
-1 & -2\\
0 & 1
\end{array}\right]$$
$$((k^{-1}_{\hat{V}}\circ \psi_{\hat{V}}\circ k_{\hat{V}})|_{\partial D_3})*=\left[\begin{array}{cc}
-1 & 1\\
0 & 1
\end{array}\right]$$

Hence we can extend across the fillings by Propositions 5.1 and 5.2.

Note that effectively, this map can be thought of as reversing the orientation of the fibers and then performing Dehn twists along the annuli $S^1 \times I$ where the intervals $I$ are shown below as bold lines:

\tikzset{every picture/.style={line width=0.75pt}} %set default line width to 0.75pt        
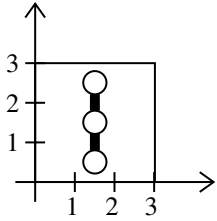
\begin{figure}[ht]
\centering
\begin{tikzpicture}[x=0.75pt,y=0.75pt,yscale=-1,xscale=1]

%Shape: Rectangle [id:dp14086199351207807] 
\draw   (139.75,90) -- (200.14,90) -- (200.14,150) -- (139.75,150) -- cycle ;
%Shape: Axis 2D [id:dp7123114601867097] 
\draw  (129.75,150) -- (229.75,150)(139.75,60) -- (139.75,160) (222.75,145) -- (229.75,150) -- (222.75,155) (134.75,67) -- (139.75,60) -- (144.75,67) (159.75,145) -- (159.75,155)(179.75,145) -- (179.75,155)(199.75,145) -- (199.75,155)(134.75,130) -- (144.75,130)(134.75,110) -- (144.75,110)(134.75,90) -- (144.75,90) ;
\draw   (166.75,162) node[anchor=east, scale=1]{1} (186.75,162) node[anchor=east, scale=1]{2} (206.75,162) node[anchor=east, scale=1]{3} (136.75,130) node[anchor=east, scale=1]{1} (136.75,110) node[anchor=east, scale=1]{2} (136.75,90) node[anchor=east, scale=1]{3} ;
%Shape: Circle [id:dp6731096328107613] 
\draw   (164,139.86) .. controls (164,136.62) and (166.62,134) .. (169.86,134) .. controls (173.09,134) and (175.71,136.62) .. (175.71,139.86) .. controls (175.71,143.09) and (173.09,145.71) .. (169.86,145.71) .. controls (166.62,145.71) and (164,143.09) .. (164,139.86) -- cycle ;
%Shape: Circle [id:dp2120951009217109] 
\draw   (164,99.86) .. controls (164,96.62) and (166.62,94) .. (169.86,94) .. controls (173.09,94) and (175.71,96.62) .. (175.71,99.86) .. controls (175.71,103.09) and (173.09,105.71) .. (169.86,105.71) .. controls (166.62,105.71) and (164,103.09) .. (164,99.86) -- cycle ;
%Shape: Circle [id:dp46496429862366195] 
\draw   (164,119.86) .. controls (164,116.62) and (166.62,114) .. (169.86,114) .. controls (173.09,114) and (175.71,116.62) .. (175.71,119.86) .. controls (175.71,123.09) and (173.09,125.71) .. (169.86,125.71) .. controls (166.62,125.71) and (164,123.09) .. (164,119.86) -- cycle ;
%Straight Lines [id:da7488537223205369] 
\draw [line width=3.75]    (169.86,125.71) -- (169.86,134) ;
%Straight Lines [id:da853215515631117] 
\draw [line width=3.75]    (169.86,114) -- (169.86,105.71) ;

\end{tikzpicture}

\caption{Dehn twist annuli}
\end{figure}

\subsection{The final involution}

We now have described involutions on $\hat{M}$ and $V(2,2;-1)$. These agree on the boundaries according to the trivial gluing. 

\subsection{Remarks on $\Psi$}

Note that in the construction of a given $\psi$, choices are made with regards to the product structures, but importantly the decomposition into $\hat{M}$ and a collection of $V(2,2;-1)$s. For example, given four critical fibers of order two, there are two possible ways to pair off the fibers. However, we will establish that $\Psi$ is in fact a conjugacy class under conjugation by fiber-preserving diffeomorphisms as Proposition 7.2.

\subsection{Corollary to Propositions 5.1 and 5.2}
\begin{cor}
    Suppose that $f$ is a fiber-preserving and orientation-reversing involution of $V(2,2;-1)$ according to some boundary product structure $k:S^1 \times S^1 \rightarrow \partial V$. Then $f$ preserves the product structure up to homology. 
\end{cor}

\begin{proof}
    We apply the two previous propositions. Take $F$ to be a disc $D$ with three discs $D_1$, $D_2$, and $D_3$ removed. Then homologically respectively take representatives $\alpha$, $\alpha_1$,$\alpha_2$, and $\alpha_3$ for the boundaries of the discs and $t$ as the representative of the fiber.

    By applying the above propositions with the fillings according to $(1,2)$, $(1,2)$, and $(-1,1)$ then  either $f$ preserves the preserves the orientation of the fibers and:

    $$f_{*}(\alpha)=f_{*}(\alpha_1^{-1}\alpha_2^{-1}\alpha_2^{-1})=f_{*}(\alpha_1^{-1})f_{*}(\alpha_2^{-1})f_{*}(\alpha_3^{-1})=\alpha_1 t \alpha_2 t \alpha_3 t^{-2} =\alpha_1 \alpha_2 \alpha_3 = \alpha ^{-1}$$

    Or $f$ reverses the orientation of the fibers and:
    $$f_{*}(\alpha)=f_{*}(\alpha_1^{-1}\alpha_2^{-1}\alpha_2^{-1})=f_{*}(\alpha_1^{-1})f_{*}(\alpha_2^{-1})f_{*}(\alpha_3^{-1})=\alpha_1^{-1} t^{-1} \alpha_2^{-1} t^{-1} \alpha_3^{-1} t^{2} =\alpha_1^{-1} \alpha_2^{-1} \alpha_3^{-1} = \alpha$$

    In both cases, $f$ preserves the product structure up to homology.
    
\end{proof}

\section{Proof of Theorem}

We first establish that all $\Psi$ contains all orientation-reversing, fiber-preserving involutions that are the identity on the base space:

\begin{prop}
    Suppose $\psi\in Diff^{fp}_-(M)$ is an involution and the identity on the base space. Then $\psi \in \Psi$.
\end{prop}

\begin{proof}
    We arbitrarily separate $M$ into $\hat{M} \cong S^1 \times F $ and a collection of $V(2,2:-1)$s.

    Then given Corollary 4.3, the product structure $k_{\hat{M}}:S^1 \times F \times \hat{M}$ is preserved up to homology on the boundary. Consequently, without loss of generality, we can assume that it is preserved on the boundary and furthermore within (as product structures are unique) $\hat{M}$ as:

    $$(k_{\hat{M}}^{-1} \circ \psi|_{\hat{M}} \circ k_{\hat{M}})(u,x)=(u^{-1},x)$$

    This satisfies 7.3.

    Now take a $V(2,2,;-1)$ and consider the loop $\sigma$ shown in gray on the base space given by the image below:

\begin{figure}[ht]
\centering

\tikzset{every picture/.style={line width=0.75pt}} %set default line width to 0.75pt        

\begin{tikzpicture}[x=0.75pt,y=0.75pt,yscale=-1,xscale=1]

%Shape: Polygon Curved [id:ds47031332169932816] 
\draw   (100.2,38.8) .. controls (150.2,29.8) and (226,69) .. (287.2,19.8) .. controls (348.4,-29.4) and (384.2,31.8) .. (398.2,49.8) .. controls (412.2,67.8) and (382.2,167.8) .. (318.2,182.8) .. controls (254.2,197.8) and (244.59,122.83) .. (193.2,115.8) .. controls (141.81,108.77) and (101.27,137.32) .. (86.2,127.8) .. controls (71.13,118.28) and (39.73,97.66) .. (42.2,77.8) .. controls (44.67,57.94) and (50.2,47.8) .. (100.2,38.8) -- cycle ;
%Curve Lines [id:da45198247646354994] 
\draw    (277.2,62.44) .. controls (312.14,82.83) and (326.88,83.68) .. (360,62.44) ;
%Curve Lines [id:da7893750396492928] 
\draw    (289.79,68.67) .. controls (317.28,55.07) and (330.36,58.76) .. (348.57,68.67) ;
%Curve Lines [id:da40440014499099053] 
\draw    (68.2,72.44) .. controls (103.14,92.83) and (117.88,93.68) .. (151,72.44) ;
%Curve Lines [id:da964922095208526] 
\draw    (80.79,78.67) .. controls (108.28,65.07) and (121.36,68.76) .. (139.57,78.67) ;
%Shape: Ellipse [id:dp08870706531179307] 
\draw   (248,123.9) .. controls (248,106.28) and (274.24,92) .. (306.6,92) .. controls (338.96,92) and (365.2,106.28) .. (365.2,123.9) .. controls (365.2,141.52) and (338.96,155.8) .. (306.6,155.8) .. controls (274.24,155.8) and (248,141.52) .. (248,123.9) -- cycle ;
%Straight Lines [id:da45129784694977737] 
\draw [line width=3]    (282.2,94.8) -- (283.2,152.8) ;
%Shape: Circle [id:dp28419019184275174] 
\draw  [fill={rgb, 255:red, 0; green, 0; blue, 0 }  ,fill opacity=1 ] (260.5,124.1) .. controls (260.5,122.11) and (262.11,120.5) .. (264.1,120.5) .. controls (266.09,120.5) and (267.7,122.11) .. (267.7,124.1) .. controls (267.7,126.09) and (266.09,127.7) .. (264.1,127.7) .. controls (262.11,127.7) and (260.5,126.09) .. (260.5,124.1) -- cycle ;
%Shape: Circle [id:dp6600937248706927] 
\draw  [fill={rgb, 255:red, 0; green, 0; blue, 0 }  ,fill opacity=1 ] (335,123.1) .. controls (335,121.11) and (336.61,119.5) .. (338.6,119.5) .. controls (340.59,119.5) and (342.2,121.11) .. (342.2,123.1) .. controls (342.2,125.09) and (340.59,126.7) .. (338.6,126.7) .. controls (336.61,126.7) and (335,125.09) .. (335,123.1) -- cycle ;
%Shape: Polygon Curved [id:ds6234994900604985] 
\draw  [line width=3]  (256.2,107.9) .. controls (259.7,104.4) and (281.7,94.7) .. (282.2,94.8) .. controls (282.7,94.9) and (282.7,150.7) .. (283.2,152.8) .. controls (283.7,154.9) and (268.7,150.9) .. (259.2,144.4) .. controls (249.7,137.9) and (247.72,130.14) .. (248,123.9) .. controls (248.28,117.66) and (252.7,111.4) .. (256.2,107.9) -- cycle ;

% Text Node
\draw (269.5,115.5) node [anchor=north west][inner sep=0.75pt]   [align=left] {{\fontfamily{pcr}\selectfont 2}};
% Text Node
\draw (345,114) node [anchor=north west][inner sep=0.75pt]   [align=left] {{\fontfamily{pcr}\selectfont 2}};
% Text Node
\draw (249,88.9) node [anchor=north west][inner sep=0.75pt]    {$\sigma $};

\end{tikzpicture}
\caption{Loop in base space}
\end{figure}
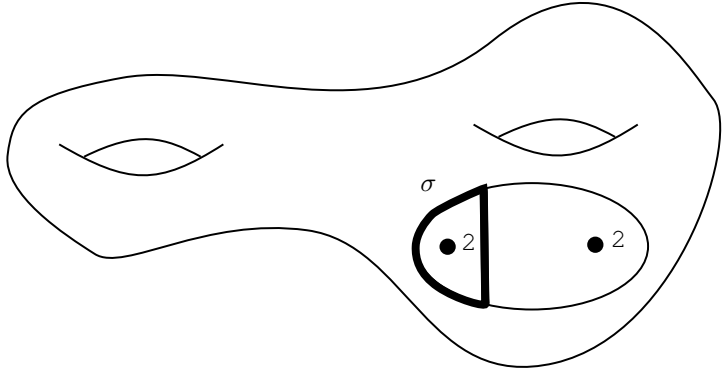

According to the product structure given for $V(2,2;-1)$ and applying Proposition 4.1 $\psi$ acts homologically on $S^1\times \sigma$ by:

$$\left[\begin{array}{cc}
-1 & 1\\
0 & 1
\end{array}\right]$$

Given that on the part of $\sigma$ on the boundary, $\psi $ acts as $(u,x)\mapsto (u^{-1},x)$, then it must be that on the remaining part parameterized as the unit interval (the orientation is determined by a positive orientation on the boundary of $V$ - in the case of the image, it would be downwards) $I$ we have $(u,t)\mapsto (u^{-1}e^{-2\pi it},t)$.

\end{proof}

\begin{prop}
$\Psi$ is a conjugacy class under conjugation by fiber-preserving diffeomorphisms.
\end{prop}

\begin{proof}
    Given $\psi\in\Psi$, we note that for a fiber-preserving diffeomorphism $g$, $g^{-1}\circ\psi \circ g$ is fiber-preserving, orientation-reversing, an involution, and the identity on the base space. Hence by Proposition 7.1, $g^{-1}\circ\psi \circ g\in\Psi$.

    This establishes that $[\psi]\subset\Psi$.

    To fully establish $[\psi]=\Psi$, we need only show that given $\psi_1,\psi_2 \in \Psi$, there exists a fiber-preserving diffeomorphism $g$ such that $\psi_1=g^{-1}\circ \psi_2 \circ g$.

    Given that $\psi_1$ and $\psi_2$ are equal on any given $V(2,2;-1)$, we construct $\hat{M}_1$ by removing all $V(2,2;-1)$s according to $\psi_1$ and similarly $\hat{M}_2$ by removing all $V(2,2;-1)$s according to $\psi_2$. We then have product structures $k_{\hat{M}_1}:S^1\times F \rightarrow \hat{M}_1$ and $k_{\hat{M}_2}:S^1\times F \rightarrow \hat{M}_2$

    There certainly exists a fiber-preserving (and orientation-preserving) diffeomorphism $\hat{g}:\hat{M}_1\rightarrow\hat{M}_2$ which is such that $k_{\hat{M}_2}^{-1}\circ \hat{g} \circ k_{\hat{M}_1}$ is a product map. This can then be extended across into the $V(2,2;-1)$s as it is homologically the identity according to the product structures induced by $k_{\hat{M}_1}$ and $k_{\hat{M}_2}$.
\end{proof}

We now establish the main result:

\begin{thm} \label{thm:mainresult}
           Let $M$ be a closed, compact, and orientable Seifert $3$-manifold that fibers over an orientable base space and is not $S^1 \times S$ for some surface $S$. Then all fiber-preserving and orientation-reversing involutions can be expressed as $\psi\circ g$ for $g$ a fiber-preserving and orientation-preserving involution $g$ and $\psi\in\Psi$.

\end{thm}

\begin{proof}
    We take our $f$ an involution which preserves fibers and reverses the orientation. We arbitrarily split into $\hat{M}$ and a collection of $V(2,2;-1)$s. By choosing a product structure $k_{\hat{M}}:S^1 \times F \rightarrow \hat{M}$ we can express $f|_{\hat{M}}$ by:

    $$(k_{\hat{M}}^{-1} \circ f|_{\hat{M}} \circ k_{\hat{M}})(u,x)=(f_1(u,x),f_2(x))$$

    We then define our $g$ by first defining it on $\hat{M}$ by:

    $$(k_{\hat{M}}^{-1} \circ g|_{\hat{M}} \circ k_{\hat{M}})(u,x)=(u^{\epsilon},f_2(x))$$

    Here $\epsilon =1$ if the fiber orientation is preserved and $\epsilon =-1$ if it is reversed.

    First note $f|_{\hat{M}} \circ g|_{\hat{M}}$ is the identity on the base space, we show that it is an involution.

    Firstly, 
    $$(k_{\hat{M}}^{-1} \circ f|_{\hat{M}} \circ g|_{\hat{M}} \circ k_{\hat{M}})(u,x)=(f_1(u^{\epsilon},f_2(x)),x)$$
    and then:
    \begin{equation}(k_{\hat{M}}^{-1} \circ (f|_{\hat{M}} \circ g|_{\hat{M}})^2 \circ k_{\hat{M}})(u,x)=(f_1(f_1(u^{\epsilon},f_2(x)),x)^{\epsilon},x)
    \label{1}
    \end{equation}

    But we have that:

    $$(u^{\epsilon},f_2(x))=(k_{\hat{M}}^{-1} \circ f|_{\hat{M}}^2 \circ k_{\hat{M}})(u^{\epsilon},f_2(x))=(f_1(f_1(u^{\epsilon},f_2(x)),x),f_2(x))$$

    So then:

    $$f_1(f_1(u^{\epsilon},f_2(x)),x)^{\epsilon}=(u^{\epsilon})^{\epsilon}=u$$

    So that in equation (1) we have:
    \begin{equation}(k_{\hat{M}}^{-1} \circ (f|_{\hat{M}} \circ g|_{\hat{M}})^2 \circ k_{\hat{M}})(u,x)=(u,x)
    \end{equation}

    That is, $f|_{\hat{M}} \circ g|_{\hat{M}}$ is an involution.

    We can certainly extend $g|_{\hat{M}}$  over to the $V(2,2:-1)$s by Propositions 4.1 and 4.2 as an involution. And so $f|_{\hat{M}} \circ g|_{\hat{M}}$ also extends as an involution that is the identity on the base space.

    By Proposition 9.1, $f \circ g = \psi$ for some $\psi \in \Psi$ and $f=\psi\circ g$.
    
\end{proof}

We here state the following Lemma:

\begin{lem}
    Let $g\in Diff^{fp}(M)$ be an involution. Then for a given $\psi\in \Psi$, $g\circ \psi$ is an involution if and only if $g$ and $\psi$ commute.
\end{lem}

\begin{proof}
    This follows from:
    $$(g \circ \psi)^2=id \iff g\circ \psi=(g\circ\psi)^{-1}=\psi^{-1} \circ g^{-1}=\psi \circ g$$
\end{proof}

We use this lemma and the main Theorem 9.2 to prove:

\begin{cor}
    There is no fiber-preserving, and orientation-reversing involution that leaves a fiber invariant and is orientation-preserving on the base space.
\end{cor}

\begin{proof}
    Suppose that such an $f$ exists. Then by Theorem 9.2, $f=g\circ \psi$ for some fiber-preserving and orientation-preserving involution $g$ and $\psi \in \Psi$.

    We show that $g$ and $\psi$ cannot commute and so by Lemma 9.3 we contradict the existence of $f$.

    We first note that on the base space, we can assume that we have the following image. Here the first disc is mapped under $g$ to the second disc so that $I_2$ is outside the first disc. If not, we can conjugate $f$ by an appropriate fiber-preserving map to make it so:

    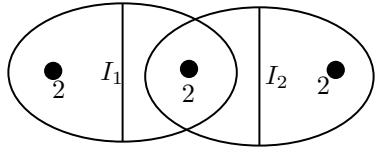
\begin{figure}[ht]
\centering

\tikzset{every picture/.style={line width=0.75pt}} %set default line width to 0.75pt        

\begin{tikzpicture}[x=0.75pt,y=0.75pt,yscale=-1,xscale=1]

%Shape: Ellipse [id:dp6637367321682435] 
\draw   (140,96.9) .. controls (140,77.63) and (165.79,62) .. (197.6,62) .. controls (229.41,62) and (255.2,77.63) .. (255.2,96.9) .. controls (255.2,116.17) and (229.41,131.8) .. (197.6,131.8) .. controls (165.79,131.8) and (140,116.17) .. (140,96.9) -- cycle ;
%Shape: Ellipse [id:dp871514305465044] 
\draw   (209,98.9) .. controls (209,79.63) and (234.79,64) .. (266.6,64) .. controls (298.41,64) and (324.2,79.63) .. (324.2,98.9) .. controls (324.2,118.17) and (298.41,133.8) .. (266.6,133.8) .. controls (234.79,133.8) and (209,118.17) .. (209,98.9) -- cycle ;
%Straight Lines [id:da39446653881311844] 
\draw    (197.6,62) -- (197.6,131.8) ;
%Straight Lines [id:da8345892367364812] 
\draw    (266.6,64) -- (266.6,133.8) ;
%Shape: Circle [id:dp8932583851964133] 
\draw  [fill={rgb, 255:red, 0; green, 0; blue, 0 }  ,fill opacity=1 ] (227,95.1) .. controls (227,92.84) and (228.84,91) .. (231.1,91) .. controls (233.36,91) and (235.2,92.84) .. (235.2,95.1) .. controls (235.2,97.36) and (233.36,99.2) .. (231.1,99.2) .. controls (228.84,99.2) and (227,97.36) .. (227,95.1) -- cycle ;
%Shape: Circle [id:dp6128489197544353] 
\draw  [fill={rgb, 255:red, 0; green, 0; blue, 0 }  ,fill opacity=1 ] (158.5,96.1) .. controls (158.5,93.84) and (160.34,92) .. (162.6,92) .. controls (164.86,92) and (166.7,93.84) .. (166.7,96.1) .. controls (166.7,98.36) and (164.86,100.2) .. (162.6,100.2) .. controls (160.34,100.2) and (158.5,98.36) .. (158.5,96.1) -- cycle ;
%Shape: Circle [id:dp41535203972873547] 
\draw  [fill={rgb, 255:red, 0; green, 0; blue, 0 }  ,fill opacity=1 ] (301,95.6) .. controls (301,93.34) and (302.84,91.5) .. (305.1,91.5) .. controls (307.36,91.5) and (309.2,93.34) .. (309.2,95.6) .. controls (309.2,97.86) and (307.36,99.7) .. (305.1,99.7) .. controls (302.84,99.7) and (301,97.86) .. (301,95.6) -- cycle ;

% Text Node
\draw (160.5,99.5) node [anchor=north west][inner sep=0.75pt]    {$2$};
% Text Node
\draw (226,101.5) node [anchor=north west][inner sep=0.75pt]    {$2$};
% Text Node
\draw (294,97.5) node [anchor=north west][inner sep=0.75pt]    {$2$};
% Text Node
\draw (185,90) node [anchor=north west][inner sep=0.75pt]    {$I_{1}$};
% Text Node
\draw (268,92) node [anchor=north west][inner sep=0.75pt]    {$I_{2}$};

\end{tikzpicture}
\caption{Base space}
\end{figure}

    Secondly, by \cite{peet2019finite} and Theorem 1.1, $g$ is an extended product involution - it fixes a fiber, so the orbit number of that fiber is 1. We therefore can assume that it acts as a product on $I_1$ and $I_2$.

    Given this, we establish product structures on $k_1:S^1\times I \rightarrow S^1\times I_1$ and $k_2:S^1\times I \rightarrow S^1\times I_2$.

    So then $(k_2^{-1} \circ g|_{I_1} \circ k_1)(u,t)=(u^\epsilon,g_2(t))$  and:

$$(k_2^{-1}\circ (g\circ \psi)|_{I_1}\circ k_1)(u,t)=(k_2^{-1}\circ g|_{I_1} \circ k_1^{-1}) \circ (k_1 \circ \psi|_{I_1}\circ k_1)(u,t)=(k_2^{-1}\circ g|_{I_1} \circ k_1^{-1})(u^{-1}e^{-2\pi it},t)=(u^{-\epsilon}e^{-2\pi \epsilon it},g_2(t))$$

On the other hand:

$$(k_2^{-1}\circ (\psi\circ g)|_{I_1}\circ k_1)(u,t)=(k_2^{-1}\circ \psi|_{I_2} \circ k_2^{-1}) \circ (k_2 \circ g|_{I_1}\circ k_1)(u,t)=(k_2^{-1}\circ \psi|_{I_2} \circ k_2^{-1})(u^\epsilon,g_2(t))=(u^{-\epsilon},g_2(t))$$

Clearly, these are not equal.

\end{proof}

Note that given no free group action is admissible, there is a fiber fixed and the obstruction condition is always satisfied.

\section{Nonorientable base space}

We once again quote from \cite{peet2019finitenon}:

\begin{thm}\label{thm:nonorientable}
 Let $M$ be an orientable Seifert fibered manifold that fibers over an orbifold $B$ that has non-orientable underlying space and $p:\tilde{M}\rightarrow M$ be the orientable base space double cover. Let the covering translation be $\tau:\tilde{M}\rightarrow\tilde{M}$. Then there exists an isomorphism between $Diff_{+}^{fp}(M)$ and $Cent_{+}^{fop}(\tau)$.
\end{thm}

We give the following corollary:

\begin{cor}
 Let $M$ be an orientable Seifert fibered manifold that fibers over an orbifold $B$ that has non-orientable underlying space and $p:\tilde{M}\rightarrow M$ be the orientable base space double cover. Let the covering translation be $\tau:\tilde{M}\rightarrow\tilde{M}$. Then there exists an isomorphism between $Diff_{-}^{fp}(M)$ and $Cent_{-}^{for}(\tau)$.
\end{cor}

\begin{proof}
    We begin by noting that we can take $\tilde{\psi} \in \tilde{\Psi}\subset Diff_{-}^{fp}(\tilde{M})$ and $\tilde{\psi}$ commutes with the covering translation $\tau$. To see this, we refer the reader to \cite{peet2019finitenon}. We can then define $\psi \in Diff_{-}^{fp}(M)$ with $\hat{\psi} \circ p = p \circ \psi$. 
   
   Then $\psi$ will also be an orientation-reversing involution that is the identity on the base space.

    Given $f\in Diff_{-}^{fp}(M)$, we have $f \circ \psi \in Diff_{+}^{fp}(M)$. So, by the theorem, there exists $\tilde{g}\in Cent_{+}^{fop}(\tau)$ such that: $p \circ \tilde{g} = f \circ \psi \circ p$.

    Now $\tilde{g} \circ \tilde{\psi} \in Cent_{-}^{for}(\tau)$ as $\tilde{g} \circ \tilde{\psi} \circ \tau = \tilde{g} \circ \tau \circ \tilde{\psi}= \tau \circ \tilde{g} \circ \tilde{\psi}$.

    Consequently, setting $\tilde{f}=\tilde{g} \circ \tilde{\psi} \in Cent_{-}^{for}(\tau)$ we have that:

    $p\circ \tilde{f} = (p \circ \tilde{g})\circ \tilde{\psi}=f \circ \psi \circ p \circ \tilde{\psi}=f \circ \psi \circ \psi \circ p = f \circ p$.

    This establishes the isomorphism.
\end{proof}

\section{Example}

We finish with a specific admissible manifold. We take the Euclidean manifold $M=(0,o_1|(2,1),(2,1),(2,1),(2,1),(1,-2))$.

We take an arbritary $\psi\in\Psi$ for this manifold. By Thereom 7.3, we know all orientation-reversing and fiber-preserving involutions will be of the form $\psi\circ g$ for orientation-preserving and fiber-preserving involutions $g$.

By the work of \cite{peet2019finite}, we know that any $g$ will be an extended product involution. In this case, we can decompose it into an orientation-preserving product involution $\hat{g}(u,x)=(g_1(u),g_2(x))$ on $S^1\times (S^1\setminus(D_1\cup D_2\cup D_3\cup D_4))$ and extend across the Dehn fillings.

We then can use \cite{dugger2019involutions} to characterize $\hat{g}$ in the two cases of:

\subsection{Fiber-orientation preserving}

We can assume that $g_1(u)=u$ (if not adjust the product structure) and by\cite{duggar} we have that $g_2$ is simply $spit_{0,0}$ (restricted to $S^1\setminus(D_1\cup D_2\cup D_3\cup D_4)$.

Note that once restricted, $spit_{0,0}$ either fixes none of the boundary components or fixes two of the boundary components. Hence up to conjugation, there are $2$ possibilities.

\subsection{Fiber-orientation reversing}

We can assume that $g_1(u)=u^{-1}$ and then by \cite{dugger2019involutions} $g_2$ is either $refl_{0,0}$ or $anti_{0,0}$ restricted to $S^1\setminus(D_1\cup D_2\cup D_3\cup D_4)$.

For both $refl$ and $anti_{0,0}$, again either none of the boundary components or two of them are fixed, once again two possibilities each, four total.

In all we see that up to conjugation there are only six possible fiber-preserving and orientation-reversing involutions

\bibliographystyle{unsrt}
\bibliography{references}

\end{document}